%% file: BS_Topological_Computational.tex
\documentclass[a4paper,11pt,english]{amsart}
\usepackage[english]{babel}
\usepackage[utf8]{inputenc}
\usepackage[T1]{fontenc}
\usepackage{graphicx}
\usepackage[bookmarks=true]{hyperref}
\usepackage{amsmath,mathrsfs,url,amsfonts}
\usepackage{amssymb,enumitem,fancyhdr}
\usepackage{latexsym}
\usepackage{amsthm}
\usepackage[all]{xy}
\usepackage{color}
\usepackage[a4paper,top=3cm, bottom=3cm, left=3cm, right=3cm]{geometry}
\usepackage{algorithm}
\usepackage{algorithmicx}
\usepackage{algpseudocode}
\usepackage{epstopdf}
\usepackage{bm}
\usepackage{hhline}
\usepackage{multirow}
\usepackage{pifont}
\theoremstyle{definition}
\newtheorem{defin}{Definition}[section]

\theoremstyle{definition}
\newtheorem{example}[defin]{Example}
\newtheorem{oss}[defin]{Remark}
\theoremstyle{plain}
\newtheorem{teor}[defin]{Theorem}
\newtheorem{lemma}[defin]{Lemma}
\newtheorem{coroll}[defin]{Corollary}
\newtheorem{pro}[defin]{Proposition}
\DeclareMathOperator{\Fe}{Ker}
\DeclareMathOperator{\im}{Im}

\DeclareMathOperator{\chara}{char}

\DeclareMathOperator{\supp}{supp}
\DeclareMathOperator{\link}{link}

\DeclareMathOperator{\Tor}{Tor}

{\left\lbrace\begin{array}{@{}l@{}}}%
{\end{array}\right.}

\newcommand{\bn}{ \beta }

\newcommand{\de}{ \partial }

\def \N {{\mathbb N}}
\def \Z {{\mathbb Z}}

\def \S {{\mathbb S}}

\def \K {{\Bbbk}}

\def \RP {{\mathbb{RP}}}
\usepackage{breakcites}

\begin{document}
\title{Betti splitting from a topological point of view}
\author{Davide Bolognini, Ulderico Fugacci}

\begin{abstract}
A Betti splitting $I=J+K$ of a monomial ideal $I$ ensures the recovery of the graded Betti numbers of $I$ starting from those of $J,K$ and $J \cap K$. In this paper, we introduce this condition for simplicial complexes, and, by using Alexander duality, we prove that it is equivalent to a recursive splitting conditions on links of some vertices. The adopted point of view enables for relating the existence of a Betti splitting for a simplicial complex $\Delta$ to the topological properties of $\Delta$. Among other results, we prove that orientability for a manifold without boundary is equivalent to admit a Betti splitting induced by the removal of a single facet. Taking advantage of this topological approach, we provide the first example in literature admitting Betti splitting but with characteristic-dependent resolution. Moreover, we introduce the notion of splitting probability, useful to deal with results concerning existence of Betti splitting.
\end{abstract}

\maketitle
\section{Introduction}\label{sec:intro}
A fundamental tool to describe the structure of a homogeneous ideal $I$ in a polynomial ring is the minimal graded free resolution of $I$ and, in particular, its graded Betti numbers $\beta_{i,j}(I)$. Dealing with ideals of large size, the retrieval of these algebraic invariants can be hard from a computational point of view, also in the case of monomial ideals. A common strategy to obtain the information on $I$ is to decompose it into smaller ideals, in order to recover the invariants of $I$ using the invariants of its pieces. Following this idea, originally introduced in \cite{eiker} and developed in \cite{ur}, a Betti splitting of a monomial ideal $I$ consists of a suitable decomposition $I=J+K$ of $I$ ensuring the complete retrieval of the graded Betti numbers of $I$ from the ones of $J$, $K$ and $J\cap K$. The decomposition $I=J + K$ is called a {\em Betti splitting} of $I$ if
\begin{align*}\label{eq:Betti}
\beta_{i,j}(I)=\beta_{i,j}(J)+\beta_{i,j}(K)+\beta_{i-1,j}(J \cap K), \text{ for all } i, j \in \mathbb{N}.
\end{align*}
In this paper, we introduce this condition from a combinatorial point of view, giving a new perspective on the topic, framing it in topological terms. 


In several applications, it is interesting to describe the topology of a geometric realization of a simplicial complex $\Delta$, in particular its homology. The remarkable Hochster's formula relates the graded Betti numbers of the Alexander dual ideal $I_{\Delta}^*$ of $\Delta$ and the reduced homology of suitable subcomplexes of $\Delta$; in this way Betti splittings are related to the classical Mayer-Vietoris approach. 

Reading the Betti splitting condition for $I_{\Delta}^*$ from a purely topological point of view, we relate topological properties and features of a geometric realization of $\Delta$ and existence of suitable Betti splittings for $I_{\Delta}^*$. In Example \ref{len}, we present the first example in literature of an ideal with characteristic-dependent resolution admitting a Betti splitting over {\em every} field, pointing out that a topological approach is a natural way to deal with such kind of problems. This answers to \cite[Question 4.3]{ur}.

In this framework, we introduce the notion of {\em homology splitting} for a simplicial complex $\Delta$, see Definition \ref{def:homspli}. It corresponds to a decomposition of $I_{\Delta}^*$ for which the previous equation on graded Betti numbers holds for $j=n$ and for every $i \in \mathbb{N}$. Using this notion, we are able to give a complete characterization of Betti splitting for simplicial complexes, pointing out also the intrinsic recursive nature of this tool, see Theorem \ref{teor:splitforcomplex}. 

Inspired by a tecnique introduced in \cite{ur2}, we study homology splittings induced by the removal of a single monomial of $I_{\Delta}^*$: it corresponds to remove a single facet from a simplicial complex $\Delta$. Starting from this, we introduce the notion of {\em essential facet} of a simplicial complex, see Definition \ref{def:essential}. Using this tool, we prove that there is a relation between the non-vanishing of top-homology and the existence of a Betti splitting (Theorem \ref{teo:1}). 

In the relevant case of simplicial manifolds without boundary, the homology splitting condition is equivalent to the Betti splitting condition, restricting to decompositions for which the intersection of the pieces is a manifold as well (Theorem \ref{teo:3}). This yields to the existence of Betti splittings for orientable simplicial manifolds (Proposition \ref{pro:essMan}) and for {\em all} simplicial manifolds if $\chara(\K)=2$ (Proposition \ref{pro:essManZ2}). Moreover we are able to characterize orientability of a manifold in terms of Betti splittings induced by the removal of a single facet (Theorem \ref{viceversa}).

In Section \ref{sec:examples}, we consider pathological simplicial complexes that do not admit Betti splitting, depending on the field. We provide obstructions to such kind of decomposition (Proposition \ref{pro:necessary}). From this result, we can test algorithmically these obstructions, at least in dimension two, see \cite{ico}.

All our result seems independent of the chosen triangulation of the considered spaces. To formalize this, we introduce the notion of {\em best Betti splitting probability}, see Definition \ref{bettiprob}. Several results of the paper can be given in terms of this notion and it can be the starting point to deal with some new problems.

\section{Preliminaries}\label{sec:back}
Let $\K$ be a field, $R=\K[x_1,\dots, x_n]$ be the polynomial ring on $n$ variables with coefficients in $\K$ and $\mathfrak{m} = (x_1, \dots, x_n)$ its maximal homogeneous ideal. Let $M$ be a finitely generated graded $R$-module. The {\em minimal graded free resolution} of $M$ as $R$-module is a free resolution of $M$ of the form
$$  0 \rightarrow  F_{p} \xrightarrow{\phi_{p}}  \cdots \xrightarrow{\phi_{2}}  F_{1} \xrightarrow{\phi_{1}}  F_{0} \xrightarrow{\phi_{0}}  M \rightarrow  0 $$ where $F_i = \bigoplus_{j\in \N} R(-j)^{\bn_{i,j}(M)}$, $\phi_i$ are homogeneous maps and $\im(\phi_i) \subseteq \mathfrak{m} F_{i-1}$ for each integer $i$. The invariants $\bn_{i,j}(M) = \dim_\K(F_i)_j =\dim_{\K} \Tor_i(M,\K)_j$ are the {\em graded Betti numbers} of $M$. Denote by $\bn_{i}(M)=\sum_{j\in \N} \bn_{i,j}(M)$ the {\em $i^{th}$ total Betti number} of $M$.

If $M$ is graded over $\mathbb{Z}^n$, we can consider its multigraded resolution and its {\em multigraded Betti numbers} $\beta_{i,\underline{a}}(M)$, where $\underline{a} \in \mathbb{Z}^n$. Typical examples of multigraded modules are given by monomial ideals. Denote by $\supp({\underline{a}})$ the set $\{i:a_i \neq 0\}$ and by $|\underline{a}|:=\sum_{i=1}^n a_i$, for every $\underline{a} \in \mathbb{Z}^n$.

Given a monomial ideal $I \subseteq R$, we denote by $G(I)$ the minimal system of monomial generators of $I$.

Betti splitting tecnique is a Mayer-Vietoris approach for the study of graded Betti numbers of monomial ideals. It appeared for the first time in \cite{eiker}, but it has been formalized in full generality only in \cite{ur}.

\begin{defin} \cite[Definition 1.1]{ur} \label{def:Ibettispli} Let $I$, $J$ and $K$ be monomial ideals in $R$ such that $I = J + K$ and $G(I)$ is the disjoint union of $G(J)$ and $G(K)$. We say that $J + K$ is a {\em Betti splitting} of $I$ over $\K$ if
$$\bn_{i,j}(I)=\bn_{i,j}(J)+\bn_{i,j}(K)+\bn_{i-1,j}(J \cap K), \text{ for all } i, j \in \N.$$
\end{defin}

This condition is equivalent to the vanishing of some maps between Tor modules, see \cite[Proposition 2.1]{ur}. In particular, the induced maps $$\Tor_i(J \cap K;\K)_j \rightarrow \Tor_i(J;\K)_j \oplus \Tor_i(K;\K)_j$$ must be zero, for all $i,j \in \N$.

Notice that these maps are zero if and only if the corresponding maps $$\Tor_i(J \cap K;\K)_{\underline{a}} \rightarrow \Tor_i(J;\K)_{\underline{a}} \oplus \Tor_i(K;\K)_{\underline{a}}$$ on multidegrees $\underline{a}$ such that $|\underline{a}|=j$ are zero. 

From this, it follows that the Betti splitting formula for degrees holds if and only if the Betti splitting formula for multidegrees holds.  

\begin{example}\label{tostart}
Let $I=(x_4x_5,x_1x_5,x_1x_3,x_1x_2) \subseteq \K[x_1,\cdots,x_5]$. Consider the decomposition $I=J+K$, with $J=(x_4x_5,x_1x_3)$ and $K=(x_1x_5,x_1x_2)$. It is not difficult to see that $\beta_{1,4}(J) \neq 0$, since $J$ is a complete intersection. But $\beta_{1,4}(I)=0$, since $I$ is an ideal with linear resolution. Then, the given decomposition is not a Betti splitting. Instead, for instance, $I=(x_4x_5,x_1x_5)+(x_1x_3,x_1x_2)$ is a Betti splitting of $I$.
\end{example}

In this paper, we introduce Betti splitting for simplicial complexes, inspired by the corresponding notion for ideals associated to them. Our purpose is to point out that a topological approach to Betti splitting problems yields to interesting results and examples. An {\em abstract simplicial complex} $\Delta$ on $n$ {\em vertices} is a collection of subsets of $[n]=\{1, \dots ,n\}$, called {\em faces}, such that if $F \in \Delta$, $G \subseteq F$, then $G \in \Delta$. A simplicial complex $\Delta$ is completely determined by the collection $\mathcal{F}(\Delta)$ of its {\em facets}, the maximal faces with respect to inclusion. Define the {\em dimension} of a facets $F$ as $\mathrm{dim}(F):=|F|-1$ and $\dim(\Delta)$ as $$\dim(\Delta):=\mathrm{max} \{\mathrm{dim}(F):F \in \mathcal{F}(\Delta\}.$$ Denote by $\langle F_1, \dots, F_r \rangle$ the simplicial complex determined by the facets $F_1, \dots, F_r$. Every abstract simplicial complex can be seen as a topological space $|\Delta|$.  


Let $k \in \mathbb{N}$. Denote by $\widetilde{H}_k(\Delta;\K)$ the $k^{th}$ simplicial reduced homology group of $\Delta$ with coefficients in $\K$. It can be proved that $\widetilde{H}_k(\Delta;\K) \simeq \widetilde{H}_k(|\Delta|;\K)$, see \cite{hatcher}.  Denote by $$\widetilde{\bn}_k(\Delta;\K):=\mathrm{dim}_{\K} \widetilde{H}_k(\Delta;\K)$$ and recall that $\widetilde{\beta}_{-1}(\Delta;\K) \neq 0$ if and only if $\Delta=\{\emptyset\}$; in this case $\widetilde{\beta}_{-1}(\Delta;\K)=1$.

There are several ways to associate a squarefree monomial ideal to a simplicial complex $\Delta$. In the literature, the most studied ideal is the so-called {\em Stanley-Reisner ideal}. In this paper, we are interested in the Alexander dual ideal of $\Delta$, that can be viewed as the Stanley-Reisner ideal of a dual simplicial complex $\Delta^*$ (see for instance \cite{hibi}).

\begin{defin}\label{def:alexi}
Let $\Delta$ be a simplicial complex on $[n]$. The {\em Alexander dual ideal} of $\Delta$ is the squarefree monomial ideal defined by $$I_{\Delta}^*=(x_{[n] \setminus F}:F \in \mathcal{F}(\Delta)) \subseteq R=\K[x_1,\cdots,x_n],$$ where $x_{[n] \setminus F}=\prod_{i \in [n] \setminus F}x_i$.
\end{defin}

\begin{oss}
Given a squarefree monomial ideal $I \subseteq R$ and fixing the number of variables of $R$, there is a {\em unique} simplicial complex $\Delta$ such that $I=I_{\Delta}^*$. {\em For the rest of the paper, complements of facets are computed with respect to {\em all} the $n$ variables of $\K[x_1,\cdots,x_n]$}.
\end{oss}

In \cite{hochster77}, M. Hochster proved a remarkable formula: it gives a relation between the homology of suitable subcomplexes of a simplicial complex and the graded Betti numbers of its associated ideals. The version of Hochster's formula that we recall here is due to Eagon and Reiner \cite{eagon98}.

Given a simplicial complex $\Delta$, we define the {\em link} of a face $F$ in $\Delta$ as $$\link_{\Delta}F:=\{G \in \Delta:F \cup G \in \Delta, F \cap G=\emptyset\}.$$ Notice that $\link_\Delta \emptyset=\Delta$.

\begin{teor}{\em (Hochster's formula)}.\label{thm:Hoch}  Let $\Delta$ be a simplicial complex on $[n]$ and $\underline{a} \in \mathbb{Z}^n$. Then $\bn_{i,\underline{a}}(I_{\Delta}^*) \neq 0$ if and only if $\supp(\underline{a})=[n] \setminus G$, for some $G \in \Delta$. 

In this case $\bn_{i,\underline{a}}(I^*_\Delta) = \widetilde{\bn}_{i-1}(\link_\Delta G; \K)$. In particular, $$\bn_{i,j}(I^*_\Delta) = \sum_{G \in \Delta, |G|=n-j} \widetilde{\bn}_{i-1}(\link_\Delta G; \K).$$
\end{teor}

\begin{oss}\label{oss:hom} Hochster's formula relates explicitly the homology of a simplicial complex $\Delta$ with the graded Betti numbers of $I^*_{\Delta}$. For $j=n$, we obtain $\bn_{i,n}(I^*_\Delta) = \widetilde{\bn}_{i-1}(\Delta; \K).$
\end{oss}
\section{Betti split\-ting for simplicial complexes}\label{sec:homSplitting}
Using Alexander dual ideals, it is possible to define the Betti splitting condition for a simplicial complex $\Delta$. The main result of this section (Theorem \ref{teor:splitforcomplex}) describe this condition recursively. Before proving it, we need some definitions.

\begin{defin} Let $\Delta$ be a simplicial complex and $\mathcal{F}(\Delta)=\mathcal{F}_1 \mathcal{t} \mathcal{F}_2$ a partition of $\mathcal{F}(\Delta)$. Let $\Delta_1=\langle F \, : \, F \in \mathcal{F}_1\rangle$ and $\Delta_2=\langle F \, : \, F \in \mathcal{F}_2\rangle$. We call $\Delta=\Delta_1\cup\Delta_2$ a {\em standard decomposition} of $\Delta$.
\end{defin}

\begin{oss}\label{disj}
The previous definition is the combinatorial counterpart of the assumption on disjoint minimal systems of generators in the definition of Betti splitting.
\end{oss}

Definition \ref{def:Ibettispli} for squarefree monomial ideals yields the following natural version for simplicial complexes.

\begin{defin}\label{def:bettispli} A standard decomposition $\Delta=\Delta_1\cup\Delta_2$ of a simplicial complex $\Delta$ on $n$ vertices is called a {\em Betti splitting} of $\Delta$ over $\K$ if $I^*_{\Delta}=I^*_{\Delta_1} + I^*_{\Delta_2}$ is a Betti splitting of $I^*_{\Delta} \subseteq \K[x_1,\cdots,x_n]$ over $\K$.

We say that $\Delta$ {\em admits a Betti splitting} over $\K$ if there exists a Betti splitting of $\Delta$ over $\K$.
\end{defin}

The following key definition is useful to state Theorem \ref{teor:splitforcomplex}.

\begin{defin}\label{def:homspli} Let $\Delta=\Delta_1\cup\Delta_2$ be a standard decomposition of a simplicial complex $\Delta$. We say that $\Delta=\Delta_1\cup\Delta_2$ is a {\em homology splitting} of $\Delta$ over $\K$ if $\Delta_1\cap\Delta_2=\emptyset$ or if
$$\widetilde{\beta}_k(\Delta;\K)=\widetilde{\bn}_k(\Delta_1; \K) + \widetilde{\bn}_k(\Delta_2; \K) + \widetilde{\bn}_{k-1}(\Delta_1\cap\Delta_2; \K), \text{ for every } k \in \mathbb{N}.$$
We say that $\Delta$ {\em admits a homology splitting} over $\K$ if there exists a homology splitting of $\Delta$ over $\K$.
\end{defin}

Notice that clearly, if $\Delta_1 \cap \Delta_2=\{\emptyset\}$, then the standard decomposition $\Delta=\Delta_1 \cup \Delta_2$ is trivially a homology splitting.

\begin{oss}
By Theorem \ref{thm:Hoch} and Remark \ref{disj}, the homology splitting formula for $\Delta$ is equivalent to the Betti splitting formula for $I_{\Delta}^*$, with $j=n$ and every $i \in \mathbb{N}$. 

Hence, $\Delta=\Delta_1 \cup \Delta_2$ is a Betti splitting of $\Delta$ implies that $\Delta=\Delta_1 \cup \Delta_2$ is a homology splitting of $\Delta$.

In general, the converse does not hold. In fact let $\Delta$ be the simplicial complex defined by $\langle 123,234,245,345 \rangle$. Clearly $\Delta=\langle 123,245\rangle \cup \langle 234,345\rangle$ is a standard decomposition of $\Delta$. It is easy to see that this is a homology splitting over every field $\K$. But it is not a Betti splitting, since $I_{\Delta}^*$ is the ideal considered in Example \ref{tostart} and the decomposition above is the algebraic counterpart of the standard decomposition given here.

In Theorem \ref{teo:3} we will describe a class of simplicial complexes and standard decompositions for which the two notions are equivalent.
\end{oss}

For sake of completeness, in the next result we give a straightforward equivalent condition to homology splitting condition.

\begin{pro}\label{hom}
Let $\Delta$ be a simplicial complex and $\Delta=\Delta_1 \cup \Delta_2$ be a standard decomposition of $\Delta$. Then, the following are equivalent:
\begin{enumerate}
\item[$\mathrm{(1)}$] $\Delta=\Delta_1 \cup \Delta_2$ is a homology splitting of $\Delta$ over $\K$;
\item[$\mathrm{(2)}$] The maps $\widetilde{H}_k(\Delta_1 \cap \Delta_2;\K) \stackrel{\phi_k}\rightarrow \widetilde{H}_k(\Delta_1;\K) \oplus \widetilde{H}_k(\Delta_2;\K)$ in the Mayer-Vietoris sequence are zero for every $k \in \mathbb{N}$.
\end{enumerate}
\end{pro}

{\em Proof.} Consider the Mayer-Vietoris exact sequence in homology arising from the decomposition $\Delta=\Delta_1 \cup \Delta_2$ (for sake of simplicity we omit the field $\K$ in the notations): $$... \rightarrow \widetilde{H}_k(\Delta_1 \cap \Delta_2) \stackrel{\phi_k}\rightarrow \bigoplus_{i=1}^2\widetilde{H}_k(\Delta_i) \stackrel{\psi_k}\rightarrow \widetilde{H}_k(\Delta) \stackrel{\partial_k}\rightarrow \widetilde{H}_{k-1}(\Delta_1 \cap \Delta_2) \stackrel{\phi_{k-1}}\rightarrow \bigoplus_{i=1}^2\widetilde{H}_{k-1}(\Delta_i) \rightarrow ...$$

First we prove that $\mathrm{(2)}$ implies $\mathrm{(1)}$. If $\Delta_1 \cap \Delta_2=\{\emptyset\}$, we have nothing to prove. Assume $\Delta_1 \cap \Delta_2 \neq \{\emptyset\}$, then $\widetilde{H}_{-1}(\Delta_1 \cap \Delta_2)=0$. From the long exact sequence above we get, for every $k \in \mathbb{N}$, \begin{align*}
\widetilde{\beta}_k(\Delta_1)+\widetilde{\beta}_k(\Delta_2) &= \dim_{\K}\Fe(\psi_k)+\dim_{\K}\im(\psi_k)=\dim_{\K}\im(\phi_k)+\dim_{\K}\im(\psi_k) \\
&= \dim_{\K}\im(\psi_k)=\dim_{\K}\Fe(\partial_k)=\widetilde{\beta}_k(\Delta)-\dim_{\K}\im(\partial_k) \\
&=\widetilde{\beta}_k(\Delta)-\dim_{\K}\Fe(\phi_{k-1})=\widetilde{\beta}_k(\Delta)-\widetilde{\beta}_{k-1}(\Delta_1 \cap \Delta_2).
\end{align*}
\indent To prove that $\mathrm{(1)}$ implies $\mathrm{(2)}$. we proceed by induction on $k \geq 0$. If $\Delta_1 \cap \Delta_2=\{\emptyset\}$, then it follows that all the maps $\phi_k=0$, for $k \in \mathbb{N}$. Assume $\Delta_1 \cap \Delta_2 \neq \{\emptyset\}$. For $k=0$ we have $\dim_{\K}\im(\phi_0)=\dim_{\K}\Fe(\phi_0)=\widetilde{\beta}_0(\Delta_1)+\widetilde{\beta}_0(\Delta_2)-\dim_{\K}\im(\psi_0)=\widetilde{\beta}_0(\Delta_1)+\widetilde{\beta}_0(\Delta_2)-\widetilde{\beta}_0(\Delta)=0$. Assume $k \geq 1$ and $\phi_{k-1}=0$; we prove that $\phi_k=0$ using an argument similar to the previous one: \begin{align*}
\dim_{\K}\im(\phi_k) &= \dim_{\K}\Fe(\psi_k)=\widetilde{\beta}_k(\Delta_1)+\widetilde{\beta}_k(\Delta_2)-\dim_{\K}\im(\psi_k)\\
&=\widetilde{\beta}_k(\Delta_1)+\widetilde{\beta}_k(\Delta_2)-\dim_{\K}\Fe(\partial_k)\\
&=\widetilde{\beta}_k(\Delta_1)+\widetilde{\beta}_k(\Delta_2)-\widetilde{\beta}_k(\Delta)+\dim_{\K}\im(\partial_k)\\
&=\widetilde{\beta}_k(\Delta_1)+\widetilde{\beta}_k(\Delta_2)-\widetilde{\beta}_k(\Delta)+\dim_{\K}\Fe(\phi_{k-1}) \\
&=\widetilde{\beta}_k(\Delta_1)+\widetilde{\beta}_k(\Delta_2)-\widetilde{\beta}_k(\Delta)+\widetilde{\beta}_{k-1}(\Delta_1 \cap \Delta_2)=0. \qed
\end{align*}

The following is an immediate corollary of the previous result.

\begin{coroll}\label{goodinter}
Let $\Delta$ be a simplicial complex and $\Delta=\Delta_1 \cup \Delta_2$ be a standard decomposition of $\Delta$. If $\Delta_1 \cap \Delta_2$ is acyclic over $\K$ then $\Delta=\Delta_1 \cup \Delta_2$ is a homology splitting of $\Delta$ over $\K$.
\end{coroll}

%

Finally, we are able to prove the main theorem of this section: a Betti splitting for $\Delta$ can be characterized in terms of homology splittings or {\em recursively} in terms of Betti splitting of vertex-links.

\begin{teor}\label{teor:splitforcomplex} Let $\Delta=\Delta_1\cup \Delta_2$ be a standard decomposition of a simplicial complex $\Delta$. Then, the following statements are equivalent:
\begin{itemize}
\item[$\mathrm{(1)}$] $\Delta=\Delta_1\cup \Delta_2$ is a Betti splitting of $\Delta$ over $\K$;
\item[$\mathrm{(2)}$] $\link_\Delta F=\link_{\Delta_1} F \cup \link_{\Delta_2} F$ is a homology splitting of $\link_\Delta F$ over $\K$, for each face $F \in \Delta_1\cap \Delta_2$;
\item[$\mathrm{(3)}$] $\link_\Delta v=\link_{\Delta_1} v \cup \link_{\Delta_2} v$ is a Betti splitting over $\K$ of $\link_\Delta v$, for each vertex $v \in \Delta_1\cap \Delta_2$ and $\Delta=\Delta_1\cup \Delta_2$ is a homology splitting of $\Delta$ over $\K$.
\end{itemize}
\end{teor}

\begin{proof}
Notice that since $\Delta=\Delta_1 \cup \Delta_2$ is a standard decomposition of $\Delta$, then $\link_{\Delta}F=\link_{\Delta_1}F \cup \link_{\Delta_2}F$ is a standard decomposition of $\link_{\Delta}F$, for every $F \in \Delta$.

First we prove that $\mathrm{(1)}$ implies $\mathrm{(2)}$. If $\Delta_1 \cap \Delta_2=\{\emptyset\}$ we have nothing to prove, since $\Delta=\Delta_1 \cup \Delta_2$ is a homology splitting of $\Delta$ over $\K$. We may assume $\Delta_1 \cap \Delta_2 \neq \{\emptyset\}$. By definition,
\begin{center}
$\beta_{i,j}(I_{\Delta}^*)=\beta_{i,j}(I_{\Delta_1}^*)+\beta_{i,j}(I_{\Delta_2}^*)+\beta_{i-1,j}(I_{\Delta_1}^* \cap I_{\Delta_2}^*)$, for every $i,j \in \mathbb{N}$.
\end{center}
Since all the maps $\Tor_i(I_{\Delta_1}^* \cap I_{\Delta_2}^*,\K) \rightarrow \Tor_i(I_{\Delta_1}^*,\K) \oplus \Tor_i(I_{\Delta_2}^*,\K)$ are zero, then all the maps $\Tor_i(I_{\Delta_1}^* \cap I_{\Delta_2}^*,\K)_{\underline{a}} \rightarrow \Tor_i(I_{\Delta_1}^*,\K)_{\underline{a}} \oplus \Tor_i(I_{\Delta_2}^*,\K)_{\underline{a}}$ are zero, for every multidegree $\underline{a} \in \mathbb{Z}^n$. Then,
\begin{center}
$\beta_{i,\underline{a}}(I_{\Delta}^*)=\beta_{i,\underline{a}}(I_{\Delta_1}^*)+\beta_{i,\underline{a}}(I_{\Delta_2}^*)+\beta_{i-1,\underline{a}}(I_{\Delta_1}^* \cap I_{\Delta_2}^*)$, for every $i \in \mathbb{N}$ and $\underline{a} \in \mathbb{Z}^n$.
\end{center}
Recall that $I_{\Delta_1}^* \cap I_{\Delta_2}^*=I_{\Delta_1 \cap \Delta_2}^*$. By Hochster's formula, $\beta_{i,\underline{a}}(I_{\Delta}^*) \neq 0$ if and only if $\supp(\underline{a})=[n] \setminus F$, for some face $F \in \Delta$; in this case $\beta_{i,\underline{a}}(I_{\Delta}^*)=\widetilde{\beta}_{i-1}(\link_{\Delta} F)$, for every $i \in \mathbb{N}$. Let $F \in \Delta_1 \cap \Delta_2$. Then, $$\widetilde{\beta}_{i-1}(\link_{\Delta} F;\K)=\widetilde{\beta}_{i-1}(\link_{\Delta_1} F;\K)+\widetilde{\beta}_{i-1}(\link_{\Delta_2} F;\K)+\widetilde{\beta}_{i-2}(\link_{\Delta_1 \cap \Delta_2} F;\K),$$ for every $i \in \mathbb{N}$.

Now we prove that $\mathrm{(2)}$ implies $\mathrm{(3)}$. Consider $F=\emptyset$. By our assumption, we have that $\Delta=\Delta_1 \cup \Delta_2$ is a homology splitting of $\Delta$ over $\K$. If $\Delta_1 \cap \Delta_2=\{\emptyset\}$, we are done. Hence assume $\Delta_1 \cap \Delta_2 \neq \{\emptyset\}$. Let $v \in \Delta_1 \cap \Delta_2$ be a vertex and $\underline{a} \in \mathbb{Z}^n$ such that $\supp(\underline{a})=[n] \setminus G$ for a suitable face $G \in \link_{\Delta}v$. If $G \notin \link_{\Delta_1}v$ (the same for $\link_{\Delta_2}v$), it follows that $\link_{\Delta}v=\link_{\Delta_2}v$ and we have nothing to prove. We may assume $G \in \link_{\Delta_1}v \cap \link_{\Delta_2}v=\link_{\Delta_1 \cap \Delta_2}v$. By definition, we have that $F=G \cup \{v\}$ is a face of $\Delta_1 \cap \Delta_2$. By $\mathrm{(2)}$, we know that $\link_{\Delta}F=\link_{\Delta_1} F \cup \link_{\Delta_2} F$ is a homology splitting of $\link_{\Delta}F$. Note that $\widetilde{\beta}_{i-1}(\link_{\Delta}F)=\widetilde{\beta}_{i-1}(\link_{\link_{\Delta}v}G)=\beta_{i,\underline{a}}(I_{\link_{\Delta}v}^*),$ over $\K$. Then, $$\beta_{i,\underline{a}}(I_{\link_{\Delta}v}^*)=\beta_{i,\underline{a}}(I_{\link_{\Delta_1}v}^*)+\beta_{i,\underline{a}}(I_{\link_{\Delta_2}v}^*)+\beta_{i-1,\underline{a}}(I_{\link_{\Delta_1 \cap \Delta_2}v}^*).$$ Summing over all $\underline{a} \in \mathbb{Z}^n$ such that $|\supp(\underline{a})|=j$, we obtain the desired splitting formula for $j$ and for every $i \in \mathbb{N}$.

Finally, we prove that $\mathrm{(3)}$ implies $\mathrm{(1)}$. Since $\Delta=\Delta_1 \cup \Delta_2$ is a homology splitting of $\Delta$, for $j=n$ the splitting formula holds. Let $j \neq n$ and $\underline{a} \in \mathbb{Z}^n$, such that $\supp(\underline{a})=[n] \setminus F$, for some face $F \in \Delta$ with $|F|=n-j$. If $F \in \Delta \setminus \Delta_1$ (the same for $\Delta_2$) we have $\link_{\Delta}F=\link_{\Delta_2}F$. Then, $\beta_{i,\underline{a}}(I_{\Delta}^*)=\beta_{i,\underline{a}}(I_{\Delta_2}^*)$ and $\beta_{i,\underline{a}}(I_{\Delta_1}^*)=\beta_{i-1,\underline{a}}(I_{\Delta_1}^* \cap I_{\Delta_2}^*)=0$. Then, we may assume $F \in \Delta_1 \cap \Delta_2$. Let $v \in F$ and $\underline{b} \in \mathbb{Z}^n$ such that $\supp(\underline{b})=[n] \setminus (F \setminus v)$. By assumption, \begin{center}
$\beta_{i,\underline{b}}(I_{\link_{\Delta}v}^*)=\beta_{i,\underline{b}}(I_{\link_{\Delta_1}v}^*)+\beta_{i,\underline{b}}(I_{\link_{\Delta_2}v}^*)+\beta_{i-1,\underline{b}}(I_{\link_{\Delta_1}v}^* \cap I_{\link_{\Delta_2}v}^*),$
\end{center}
for every $i \in \mathbb{N}$. By Hochster's formula, $$\beta_{i,\underline{b}}(I_{\link_{\Delta}v}^*)=\widetilde{\beta}_{i-1}(\link_{\link_{\Delta}v}(F \setminus \{v\}))=\widetilde{\beta}_{i-1}(\link_{\Delta} F)=\beta_{i,\underline{a}}(I_{\Delta}^*).$$ Applying the previous argument to $\Delta_1$, $\Delta_2$ and $\Delta_1 \cap \Delta_2$, we get $$\beta_{i,\underline{a}}(I_{\Delta}^*)=\beta_{i,\underline{a}}(I_{\Delta_1}^*)+\beta_{i,\underline{a}}(I_{\Delta_2}^*)+\beta_{i-1,\underline{a}}(I_{\Delta_1}^* \cap I_{\Delta_2}^*).$$ Summing over all $\underline{a} \in \mathbb{Z}^n$ such that $|\supp(\underline{a})|=j$ we conclude the proof.
\end{proof}

Apart its intrinsic interest, using Theorem \ref{teor:splitforcomplex}, it is possible to write down an algorithm to test if a given standard decomposition of a simplicial complex is a Betti splitting over a given field $\K$, see \cite{ico}.

\section{Essential facets}\label{subsec:simplEss}\mbox{}
This section is devoted to prove Theorem \ref{teo:1}, ensuring that if $\Delta$ has not trivial top homology over $\K$, then it admits a homology splitting over $\K$: it will be useful in the following. In order to do this, we will consider a special class of decompositions for a simplicial complex. 

In \cite{ur2}, the authors consider a special kind of decomposition for a monomial ideal $I$. Assume that $I$, $J$ and $K$ are monomial ideals as in Definition \ref{def:Ibettispli} and $K=(m)$ is generated by a single monomial $m$. From the simplicial complexes point of view, this corresponds to the removal of a single facet from a simplicial complex.
 
%
The following result is very useful for our purposes.

\begin{pro}{\em (\cite{delfinado}, Section 3)}.\label{edel}
Let $\Delta$ be a simplicial complex, $F \in \mathcal{F}(\Delta)$ be a facet of $\Delta$ of dimension $d$ and $\K$ be a field. Then $F$ belongs to a $d$-cycle of $\Delta$ if and only if
$$\widetilde{\bn}_k(\Delta\setminus\{F\}; \K)=\begin{cases} \widetilde{\bn}_k(\Delta; \K) -1 & \text{ if } k=d \\ \widetilde{\bn}_k(\Delta; \K) & \text{ otherwise}
\end{cases}$$
\end{pro}

Proposition \ref{edel} ensures that removing a facet of dimension $d$ from a $d$-cycle of $\Delta$ only affects the $d$-{th} homology group of $\Delta$. This suggests to remove a facet from a $d$-cycle, if available, in order to verify easily the homology splitting condition. 

All these considerations yields to the following definition. 

\begin{defin}\label{def:essential} Let $\Delta$ be a simplicial complex and let $F \in \mathcal{F}(\Delta)$ be a $d$-dimensional facet in $\Delta$.
We call $F$ an {\em essential} facet of $\Delta$ with respect to a field $\K$ if and only if $F$ belongs to a $d$-cycle of $\Delta$.

We denote by $\mathcal{E}_d(\Delta,\K)$ the collection of $d$-dimensional essential facets of $\Delta$ with respect to $\K$.
\end{defin}

\begin{lemma}\label{lem1} Let $\Delta$ be a simplicial complex. If $\mathcal{E}_d(\Delta,\K) \neq \emptyset$ then $\widetilde{\beta}_d(\Delta;\K) \neq 0$. Also the converse holds if $\mathrm{dim}(\Delta)=d$. 
\end{lemma}

\begin{proof}
The first part is trivial by definition. Then, setting $d=\mathrm{dim}(\Delta)$, we have to prove only that $\widetilde{\beta}_d(\Delta;\K) \neq 0$ implies $\mathcal{E}_d(\Delta,\K) \neq \emptyset$. Let $(\Delta_i)_{i=0, \dots, M}$ be a collection of simplicial complexes such that:
\begin{itemize}
\item $\Delta_0:=\{\emptyset\}$, $\Delta_M:=\Delta$;
\item for each $i\in \{0, \dots , M-1\}$, $\Delta_{i+1}=\Delta_i \cup \{F_{i+1}\}$, where $F_{i+1}$ is a single face of $\Delta$.
\end{itemize}
Given such a collection, let $j:=\min\{i \in \{0,\cdots,M\}:\widetilde{\bn}_d(\Delta_i; \K)\neq 0\}$. We set $F:=F_j$, the face introduced in $\Delta_j$. Since $\widetilde{\bn}_d(\Delta_{j}; \K)\neq 0$, $F$ is a facet of dimension $d$ in $\Delta$ which belongs to a $d$-cycle of $\Delta_j$. Then $F \in \mathcal{E}_d(\Delta,\K)$.
\end{proof}

In the next example, we show that in general the converse of Lemma \ref{lem1} does not hold. 

\begin{example}\label{stella}
The condition $\widetilde{\bn}_d(\Delta; \K)\neq 0$ does not imply $\mathcal{E}_d(\Delta,\K)$, if $d \neq \dim(\Delta)$. In fact, consider the simplicial complex $\Delta=\langle 123,345,246\rangle$. It is immediate to see that $\widetilde{\beta}_1(\Delta; \K) \neq 0$, for every field $\K$, but there are no facets of dimension $1$.
\end{example}


In the next result, we prove that an essential face over $\K$ induces a homology splitting over $\K$.

\begin{pro}\label{pro:essential} Let $\Delta$ be a simplicial complex. Assume $\mathcal{E}_d(\Delta,\K) \neq \emptyset$ and let $F \in \mathcal{E}_d(\Delta,\K)$. Then $\Delta=\langle G \in \mathcal{F}(\Delta):G \neq F\rangle \cup \langle F\rangle$ is a homology splitting of $\Delta$ over $\K$.
\end{pro}

\begin{proof}
For simplicity, in the proof we omit the field $\K$ from the notations. Consider the simplicial complexes
\begin{itemize}
\item $\Delta_1:=\langle G \, | \, G \in \mathcal{F}(\Delta), G \neq F \rangle$,
\item $\Delta_2:=\langle F\rangle$.
\end{itemize}
We claim that $\Delta=\Delta_1 \cup \Delta_2$ is a standard decomposition of $\Delta$.

The inclusion "$\subseteq$" is trivial. Conversely, it suffices to show that for every $(d-1)$-face $H \subseteq F$ we have $H \in \langle G \, | \, G \in \mathcal{F}(\Delta), G \neq F \rangle$. Since the facet $F$ is essential, $F$ belongs to a $d-$cycle in $\Delta$ and so $\de F$ is a boundary in $\Delta \setminus \{F\}$, where $\de$ denotes the boundary map in the chain complex of $\Delta$. Then, since $H \in \de F$, there exists a facet $G \in \mathcal{F}(\Delta)$, $G \neq F$, such that $H \subseteq G$.

Since $F \in \mathcal{E}_d(\Delta,\K)$, by Proposition \ref{edel}, we have that $$\widetilde{\bn}_k(\Delta_1)=\begin{cases} \widetilde{\bn}_k(\Delta) -1 & \text{ if } k=d \\ \widetilde{\bn}_k(\Delta) & \text{ otherwise}
\end{cases}$$\\
Moreover, $\widetilde{\bn}_k(\Delta_2)=0$, for $k \in \mathbb{N}$.

If $d=0$, we have $\Delta_1\cap\Delta_2=\{\emptyset\}$ and we have nothing to prove. Assume $d \geq 1$. Then, $\Delta_1 \cap \Delta_2$ is homeomorphic to the $(d-1)$-sphere $\S^{d-1}$. Therefore, we have $\widetilde{\bn}_{d-1}(\Delta_1\cap\Delta_2)=1$ and $\widetilde{\bn}_k(\Delta_1\cap\Delta_2)=0$ for $k \neq d-1$. It is easy to check that the above Betti numbers satisfy the equations described in Definition \ref{def:homspli}. In fact,\\
$\widetilde{\bn}_k(\Delta_1)+\widetilde{\bn}_k(\Delta_2)+\widetilde{\bn}_{k-1}(\Delta_1 \cap \Delta_2)=\widetilde{\bn}_k(\Delta)+0+0=\widetilde{\bn}_k(\Delta)$ if $k \neq d$ and\\
$\widetilde{\bn}_d(\Delta_1)+\widetilde{\bn}_d(\Delta_2)+\widetilde{\bn}_{d-1}(\Delta_1 \cap \Delta_2)=\widetilde{\bn}_d(\Delta)-1+0+1=\widetilde{\bn}_d(\Delta)$.
\end{proof}



\begin{oss}\label{dcycle}
It is worth of mention that the homology splitting in Proposition \ref{pro:essential} is ensured for {\em every} facet of the $d$-cycle considered.
\end{oss}

\begin{oss}
Clearly in general the existence of a homology splitting $\Delta=\langle G \in \mathcal{F}(\Delta):G \neq F\rangle \cup \langle F\rangle$ over $\K$ induced by a $d$-dimensional facet $F \in \mathcal{F}(\Delta)$ does not imply $\widetilde{\beta}_d(\Delta;\K) \neq 0$ or $F \in \mathcal{E}_d(\Delta,\K)$. Consider for instance the decomposition $\langle 123,345\rangle \cup \langle 246\rangle$ of the complex in Example \ref{stella}. By Corollary \ref{goodinter}, it is a homology splitting over every field $\K$, but we have $\widetilde{\beta}_2(\Delta;\K)=0$ and clearly also $\mathcal{E}_d(\Delta,\K)=\emptyset$.
\end{oss}

The following Theorem is an immediate consequence of Lemma \ref{lem1} and Proposition \ref{pro:essential}.

\begin{teor}\label{teo:1} Let $\Delta$ be a simplicial complex of dimension $d$. If $\widetilde{\bn}_d( \Delta; \K)\neq 0$, then $\Delta$ admits a homology splitting over $\K$.
\end{teor}

This result is useful because suggest us that pathological examples of simplicial complexes that {\em do not admit Betti splitting} could be found more easily among simplicial complexes with {\em trivial top homology}, see Section \ref{sec:examples}.
\section{Betti splitting of manifolds}\label{manifolds}
In this section, we apply our results to triangulations of manifolds, pointing out that our topological approach yields new results and examples, see also Section \ref{sec:examples}. Let us recall briefly some basic definitions.

\begin{defin}
A {\em topological $d$-manifold} $M$ is a Hausdorff space such that every point $x\in X$ has a neighbourhood which is homeomorphic to the $d$-dimensional Euclidean space.
\end{defin}

{\em If not differently stated, all the manifolds that we consider are compact and without boundary}. A {\em triangulation} of a manifold $M$ is a simplicial complex $\Delta$ such that $|\Delta| \cong M$. In the following, we will simply say that $\Delta$ itself is a {\em $d$-manifold}. We summarize some properties of a $d$-manifold $\Delta$:
\begin{itemize}
\item $\Delta$ is {\em pure} (all the facets of $\Delta$ have the same dimension $d$);
\item $\Delta$ is a {\em pseudomanifold} (every  $(d-1)$-dimensional face lies in exactly two facets);
\item $\link_{\Delta}v$ is a $(d-1)$-dimensional {\em homology sphere} (a topological manifold with the same homology of the sphere, for every field $\K$).
\end{itemize}

The study of Betti splittings of manifolds is interesting to relate topological properties on $\Delta$ and existence of particular standard decompositions on $\Delta$, not to compute the graded Betti numbers of $I_{\Delta}^*$, because they are essentially known, see Remark \ref{fvect}.

\begin{oss}\label{fvect}
Let $\Delta$ be a $d$-manifold. In this case $I_{\Delta}^*$ is generated in degree $n-d-1$.

The graded Betti numbers of $I_{\Delta}^*$ are essentially known. In fact $$\beta_i(I_{\Delta}^*)=f_{d-i}(\Delta)+\widetilde{\beta}_{i-1}(\Delta;\K),$$ for $0 \leq i \leq d+1$, where $(f_{-1}(\Delta),f_0(\Delta),\cdots,f_d(\Delta))$ is the {\em f-vector} of $\Delta$ and $f_i(\Delta)$ denotes the number of $i$-dimensional faces of $\Delta$. The equality above follows immediately from Hochster's formula, since $\link_{\Delta}(F)$ is a $(j-1)$-dimensional homology sphere, for every face $F \neq \emptyset$, $F \in \Delta$ and $|F|=n-j$.
\end{oss}

In the next result, we prove that for manifolds, the Betti splitting condition is equivalent to the homology splitting condition, if we assume that the intersection is a manifold as well.

\begin{teor}\label{teo:3} Let $d \geq 1$ and $\Delta$ be a connected $d$-manifold. Assume that $\Delta=\Delta_1 \cup \Delta_2$ is a standard decomposition of $\Delta$ such that $\Delta_1 \cap \Delta_2$ is a $(d-1)$-manifold. Then the following statements are equivalent:
\begin{itemize}
\item[$\mathrm{(1)}$] $\Delta=\Delta_1 \cup \Delta_2$ is a homology splitting of $\Delta$ over $\K$;
\item[$\mathrm{(2)}$] $\Delta=\Delta_1 \cup \Delta_2$ is a Betti splitting of $\Delta$ over $\K$.
\end{itemize}
\end{teor}

\begin{proof}
Since a Betti splitting is always a homology splitting, it is enough to prove that $\mathrm{(1)}$ implies $\mathrm{(2)}$. We prove the claim by induction on the dimension $d$ of the manifold $\Delta$. For $d=1$, $\Delta$ is a cycle graph. Since $\Delta=\Delta_1 \cup \Delta_2$ is a homology splitting, $\Delta_1$ and $\Delta_2$ are paths. Moreover $\Delta_1 \cap \Delta_2$ is given by a pair of points and the statement follows easily.

Assume the implication proved up to dimension $k$ and we prove it for dimension $d=k+1$. Let $v$ be a vertex of $\Delta_1\cap \Delta_2$. Notice that $\link_{\Delta_1}v$ and $\link_{\Delta_2}v$ are $\K$-acyclic, because $v \in \partial(\Delta_1) \cap \partial(\Delta_2)$, where $\partial$ denotes the boundary.

Recall that $\link_\Delta v$ is a $k$-dimensional homology sphere. By our assumption, $$\link_{\Delta_1} v \cap \link_{\Delta_2} v=\link_{\Delta_1 \cap \Delta_2} v$$ is a homology sphere as well, hence it is a $(d-2)$-manifold.
It is clear that $\link_\Delta v=\link_{\Delta_1} v \cup \link_{\Delta_2} v$ is a standard decomposition of $\link_{\Delta}(v)$. By our considerations, this is a homology splitting of $\link_\Delta v$ over $\K$.  Then by induction, $\link_\Delta v=\link_{\Delta_1} v \cup \link_{\Delta_2} v$ is also a Betti splitting of $\link_\Delta v$. Hence, by Theorem \ref{teor:splitforcomplex}, we have that the homology splitting $\Delta=\Delta_1\cup\Delta_2$ is a Betti splitting of $\Delta$.
\end{proof}

Noticing that the removal of every facet from a $d$-manifold yields to a natural standard decomposition, the following result is an immediate corollary of Theorem \ref{teo:3}.

\begin{coroll}\label{cor:3} Let $\Delta$ be a connected $d$-manifold. Given a facet $F \in \mathcal{F}(\Delta)$, the following statements are equivalent:
\begin{itemize}
\item[$\mathrm{(1)}$] $\Delta=\langle G| G \in \mathcal{F}(\Delta), G \neq F \rangle \cup \langle F\rangle$ is a homology splitting of $\Delta$ over $\K$;
\item[$\mathrm{(2)}$] $\Delta=\langle G| G \in \mathcal{F}(\Delta), G \neq F \rangle \cup \langle F\rangle$ is a Betti splitting of $\Delta$ over $\K$.
\end{itemize}
\end{coroll}

In the last part of this section, we focus on orientable manifolds.

A $d$-manifold $M$ is called {\em orientable} if it has a global consistent choice of orientation. For the formal definition and other details on orientability, we refer to \cite[Chapter 3.3]{hatcher}. 

Poincar\'{e} Duality is a fundamental result on the topology of orientable $d$-manifold. We state it in a classical fashion (for a modern treatment see \cite[Theorem 3.30]{hatcher}). Recall that with coefficients in a field, homology and cohomology are isomorphic, see \cite[Theorem 45.8]{munkres}.

\begin{teor}{\em (Poincar\'{e} Duality)}.\label{thm:poi}
Let $\Delta$ be an orientable $d$-manifold. Then, for any integer $k$ and any field $\K$, $$H_k(\Delta;\K)\cong H_{d-k}(\Delta;\K),$$ where $H_k(\Delta;\K)$ denotes the $k^{th}$ {\em non}-{\em reduced} homology group of $\Delta$ with coefficients in $\K$.\end{teor}

\begin{oss}\label{oss:poi2}
For every $d$-manifold, the isomorphism of Theorem \ref{thm:poi} holds for $\K=\Z_2$.
\end{oss}

Using Poincar\'e duality, we prove that the facets of an orientable manifold are all essential.

\begin{pro}\label{pro:essMan} Let $\Delta$ be an orientable manifold. Then $\mathcal{F}(\Delta)=\mathcal{E}_d(\Delta,\K)$, for every field $\K$. Moreover $\Delta=\langle G| G \in \mathcal{F}(\Delta), G \neq F \rangle \cup \langle F\rangle$ is a Betti splitting over $\K$, for every $F \in \Delta$.
\end{pro}

\begin{proof} Poincar\'{e} Duality (Theorem \ref{thm:poi}) ensures that $H_d(\Delta; \K)\cong H_0(\Delta;\K) \neq 0$. Furthermore, the geometrical realizations of the $d$-cycles of $H_d(\Delta; \K)$ coincide with the connected components of $\Delta$ itself. So, given any facet $F$ in $\Delta$, $F$ belongs to a $d$-cycle of $\Delta$. Then, by Proposition \ref{pro:essential} and Corollary \ref{cor:3} we conclude.\end{proof}

By Remark \ref{oss:poi2} and Proposition \ref{pro:essMan}, we obtain the following result.

\begin{pro}\label{pro:essManZ2}
Let $\Delta$ be a manifold. Then $\Delta=\langle G| G \in \mathcal{F}(\Delta), G \neq F \rangle \cup \langle F\rangle$ is a Betti splitting over $\Z_2$, for every $F \in \mathcal{F}(\Delta)$.
\end{pro}

In the following example, we show that Proposition \ref{pro:essManZ2} does not hold for manifolds with boundary.

\begin{example}
It can be proved that the standard triangulation of Rudin's ball does not admit a Betti splitting obtained removing a single facet. However, it admits another kind of decomposition that is a Betti splitting, see \cite[Example 4.7]{bolo}.
\end{example}

Finally, we characterize orientability of manifolds in terms of existence of Betti splitting induced by the removal of a single facet.

\begin{teor}\label{viceversa}
Let $d \geq 1$. For a $d$-manifold $\Delta$, the following are equivalent:
\begin{enumerate}
\item[$\mathrm{(1)}$] $\Delta$ is orientable;
\item[$\mathrm{(2)}$] $\Delta=\langle G| G \in \mathcal{F}(\Delta), G \neq F \rangle \cup \langle F\rangle$ is a Betti splitting over every field $\K$, for every facet $F \in \mathcal{F}(\Delta)$.
\item[$\mathrm{(3)}$] $\Delta=\langle G| G \in \mathcal{F}(\Delta), G \neq F \rangle \cup \langle F\rangle$ is a Betti splitting over every field $\K$, for some facet $F \in \mathcal{F}(\Delta)$.
\end{enumerate}
\end{teor}

\begin{proof}
By Proposition \ref{pro:essMan}, $\mathrm{(1)}$ implies $\mathrm{(2)}$. Clearly $\mathrm{(2)}$ implies $\mathrm{(3)}$. Then we have to prove only that $\mathrm{(3)}$ implies $\mathrm{(1)}$. 

Let $\K$ be a field such that $\chara(\K) \neq 2$. By contradiction $\Delta$ is not orientable, then $\widetilde{H}_d(\Delta;\K)=0$. But this cannot be, because $\beta_{d-1}(\langle G| G \in \mathcal{F}(\Delta), G \neq F \rangle \cap \langle F\rangle)=1$ and hence $\beta_d(\Delta)>0$.
\end{proof}
\section{Applications}\label{sec:examples}
The results of the previous sections describe a large class of discretized topological spaces admitting a Betti splitting. In this section, we present several interesting examples that do not admit Betti splitting and a general framework to find these kind of examples, taking advantage of our topological approach.

\begin{example} (Orientable Manifolds)
Theorem \ref{viceversa} ensures us that {\em every} triangulation of an orientable manifold, for instance an $n$-sphere $\S^n$, a torus $T$ and a projective plane of odd dimension $\RP^{2n+1}$, admits a Betti splitting decomposition over every field $\K$, induced by the removal of any of its top dimensional simplices. Another interesting example is provided by the triangulation $\Delta$ given in \cite{lutz} of the lens space $L(3,1)$ (for more details see Example \ref{len}).
\end{example}

\begin{example} (Non-Orientable Manifolds)
In view of Theorem \ref{viceversa}, we know that every triangulation of some relevant non-orientable manifolds, such as a projective space of even dimension $\RP^{2n}$ and the Klein bottle $K$ do not admit Betti splitting induced by the removal of a single facet if $\mathrm{char}(\K) \neq 2$. Indeed, we are able to show that suitable triangulations of these spaces (see Figure \ref{fig:examples}(a) and Figure \ref{fig:examples}(b)) do not admit {\em any} Betti splitting over a field $\K$ with $\mathrm{char}(\K) \neq 2$. This extends \cite[Example 4.1]{ur}.
\end{example}

\begin{example} (Moore space)
Theorem \ref{teo:1} ensures the existence of a homology splitting also for {\em non-manifold} simplicial complexes. Consider the mod 3 Moore space $M$ \cite{connon} depicted in Figure \ref{fig:examples}(c). $M$ is a 2-dimensional simplicial complex with $\widetilde{\beta_2}(M;\Z_3) \neq 0$. By Theorem \ref{teo:1}, it admits a homology splitting over $\Z_3$. In this case, all the facets of $M$ induce a homology splitting over $\Z_3$; it can be easily proved that it is also a Betti splitting. The situation is completely different if $\chara(\K) \neq 3$.
\end{example}

\begin{example} (Dunce hat)
For the dunce hat $D$, see Figure \ref{fig:examples}(d) the pathology is similar to the case of non-orientable manifolds, but {\em over every field}. The first author already proved in \cite{bolo} that the given triangulation of the dunce hat does not admit Betti splitting.
\end{example}

The spaces considered and their topological properties are summarized in Table \ref{tab:examples}.

\input{table1-nuova.tex}

The notion that we are going to introduce is inspired by Theorem \ref{teo:1}. It allows us to detect a large class of simplicial complexes that do not admit Betti splitting, i.e. every possible standard decomposition is not a Betti splitting.

\begin{defin}\label{def:non-trivially}
  Let $\K$ be a field and let $\Delta$ be a simplicial complex of dimension $d$. $\Delta$ is called {\em trivially decomposable} over $\K$ if there exists a standard decomposition $\Delta_1\cup \Delta_2$ of $\Delta$ such that $\widetilde{\beta}_{d-1}(\Delta_1 \cap \Delta_2; \K)=0$.
\end{defin}

As an immediate consequence of Definition \ref{def:homspli}, we can state the following proposition. It is extremely useful to construct pathological examples of ideals that do not admit Betti splitting.

\begin{pro}\label{pro:necessary}
Let $\Delta$ be a simplicial complex of dimension $d \geq 2$ with $\widetilde{\beta}_d(\Delta;\K)=0$. 

If $\Delta$ admits a homology splitting over $\K$, then it is trivially decomposable over $\K$. 

In particular, if $\Delta$ does admit Betti splitting over $\K$, then it is trivially decomposable over $\K$.
\end{pro}

Given a simplicial complex $\Delta$ of dimension $d>1$ with $\widetilde{\beta}_d(\Delta;\K)=0$, to check if $\Delta$ is {\em not} trivially decomposable over $\K$ can be tested algorithmically, performing all the possible standard decompositions and checking whether $\widetilde{\bn}_{d-1}$ of the intersection of the two pieces is zero or not. A version of this algorithm for 2-dimensional simplicial complexes has been developed and implemented in Python. The source code of this tool and of the other algorithms described in the paper can be found in \cite{ico}. In this algorithm, we take advantage of the fact that $\Delta_1 \cap \Delta_2$ has dimension $1$, i.e. it is a graph.

Using Proposition \ref{pro:necessary} and this algorithm, we prove that a Betti splitting is not available for several simplicial complexes $\Delta$, considering fields $\K$ for which $\widetilde{\beta}_2(\Delta;\K)=0$ and proving that they are {\em not} trivially decomposable. The considered spaces are depicted in Figure \ref{fig:examples}.

In Table \ref{tab:2}, for each simplicial complex $\Delta$, $n_V$ denotes the number of vertices of the chosen triangulation. In the third column, the number of standard decompositions that have to be checked is showed, while the last column shows the required time (in seconds) to perform the entire computation and to check if $\Delta$ is {\em not} trivially decomposable.
The hardware configuration used for these experiments is an Intel i7 6700K CPU at 4.00Ghz with 64 GB of RAM.

\input{table2-nuova.tex}

\begin{figure}[!htb]
  \centering
	\begin{tabular}{c c}
			\includegraphics[width=0.35\linewidth]{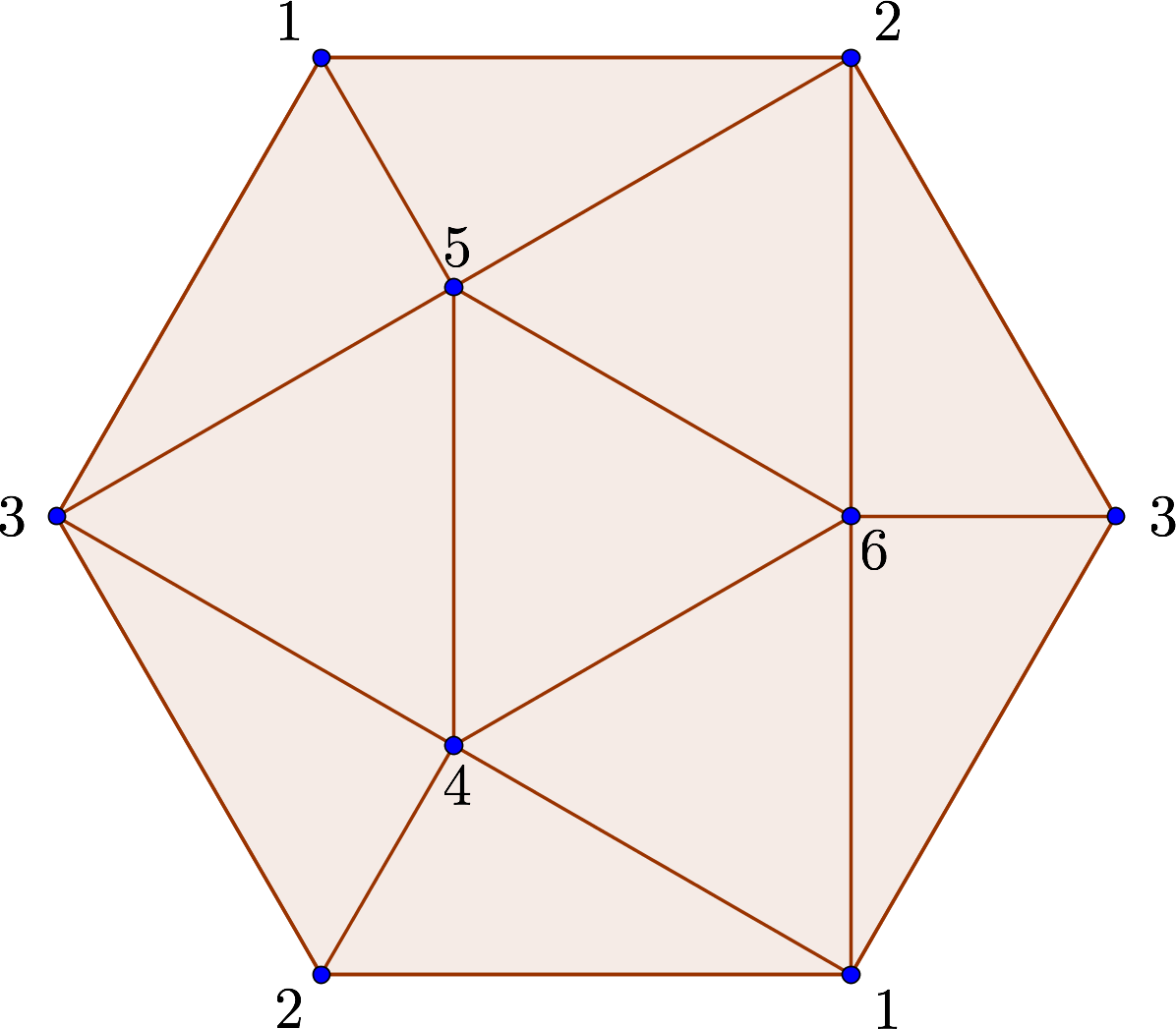} &
			\includegraphics[width=0.35\linewidth]{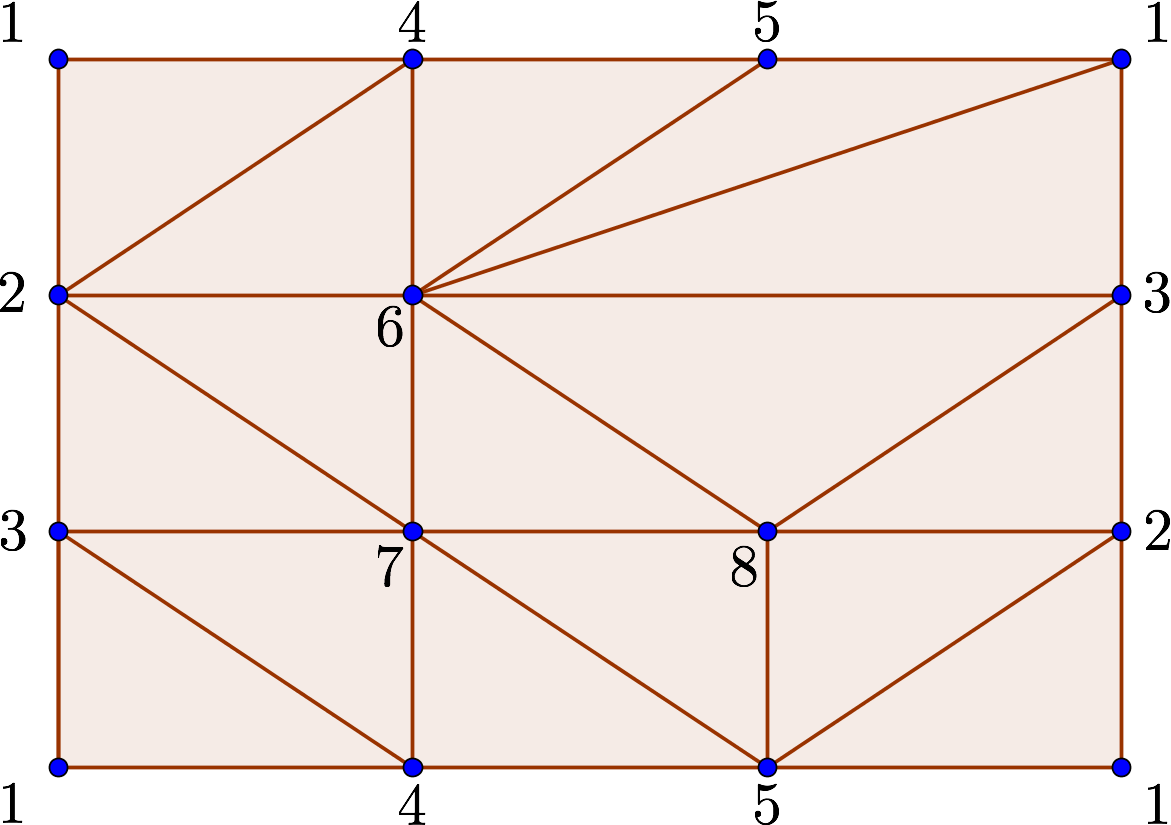}\\
			(a) & (b)\\
      \mbox{} & \mbox{} \\
      \includegraphics[width=0.35\linewidth]{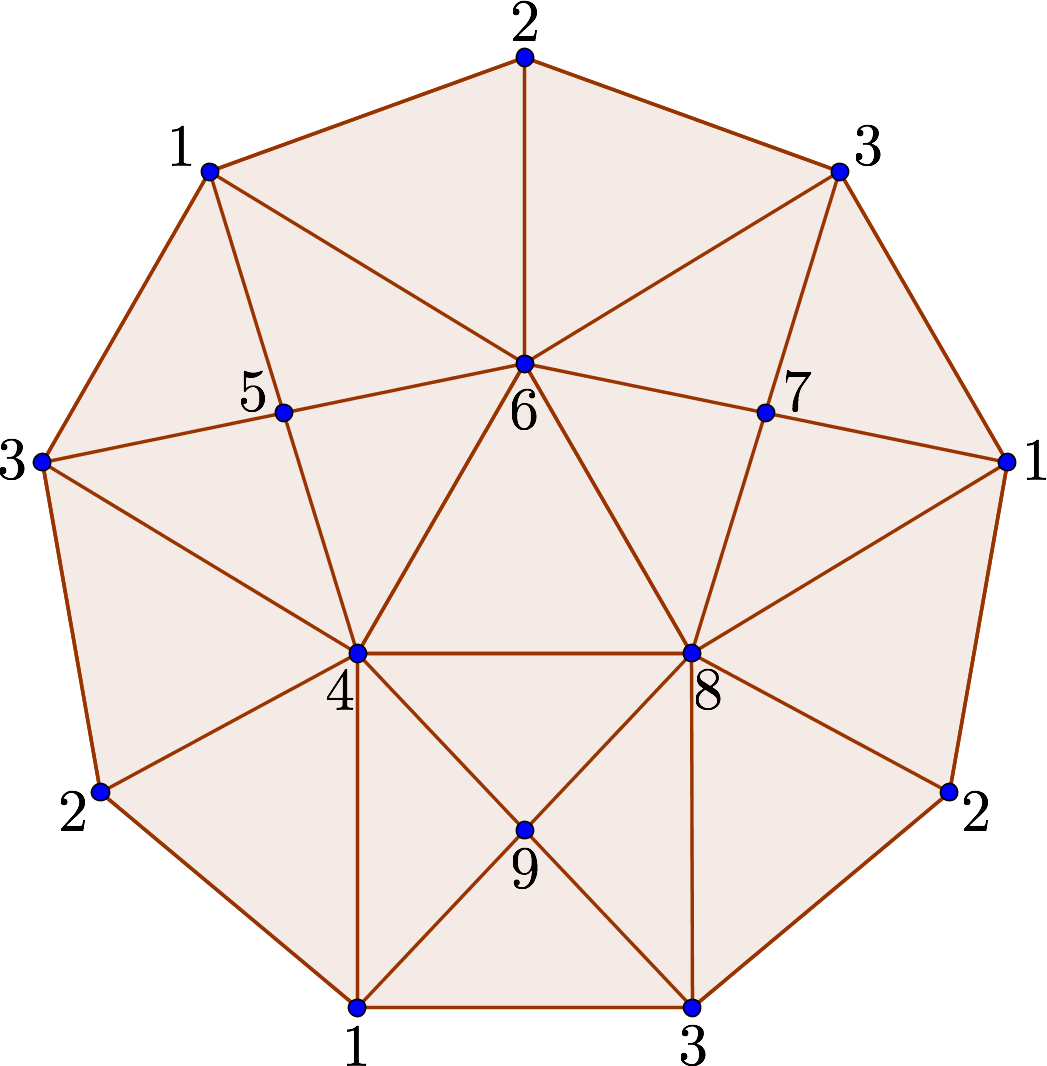} &
      \includegraphics[width=0.35\linewidth]{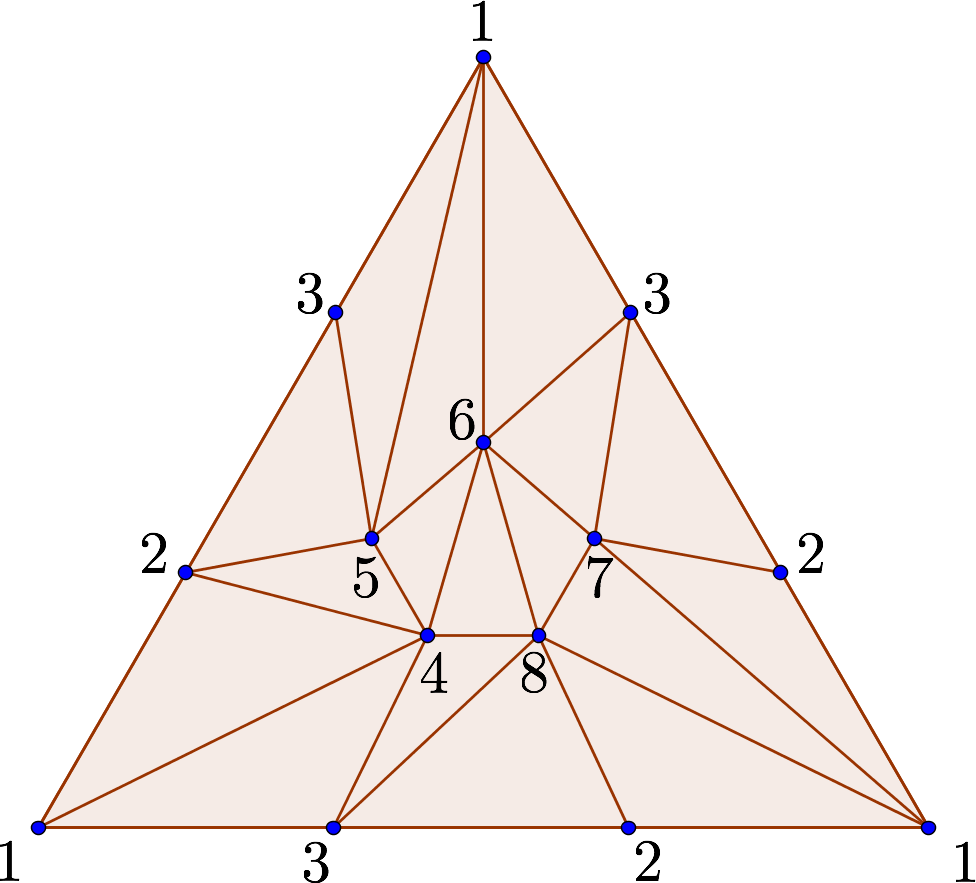}\\
      (c) & (d)\\
	\end{tabular}
  \caption{Simplicial complexes triangulating the real projective plane $\RP^2$ (a), the Klein bottle $K$ (b), the mod 3 Moore space $M$ (c) and the Dunce hat $D$ (d), respectively.}
  \label{fig:examples}
\end{figure}

\begin{example}\label{len}
Consider the triangulation $\Delta$ given in \cite{lutz} of the lens space $L(3,1)$ (for more details on this space see \cite{hatcher}). It is an orientable $3$-manifold. Then, by Theorem \ref{viceversa}, it admits Betti splitting induced by the removal of any of its facets. Its Alexander dual ideal $I_{\Delta}^*$ is the first example in literature of an ideal with characteristic dependent resolution admitting Betti splitting over every field. This answers to Question \cite[Question 4.3]{ur}.
\end{example}

Proposition \ref{pro:necessary} and the structure of the spaces considered, suggests that the pathology of some examples of this section does {\em not} depend on the chosen triangulation. To formalize this statement we introduce the following definition.

\begin{defin}\label{bettiprob}
Let $\Delta$ be a simplicial complex. Denote by $\mathcal{S}_{\Delta}$ and $\mathcal{B}_{\Delta}$ the collection of standard decompositions of $\Delta$ and the collection of these decompositions that are Betti splitting, respectively. We can define the {\em Betti splitting probability} of $\Delta$ over $\K$ as follows: $$P_{Betti}(\Delta;\K):=\frac{|\mathcal{B}_{\Delta}|}{|\mathcal{S}_{\Delta}|}.$$ For a topological space $X$ admitting a triangulation, we can define the {\em best Betti splitting probability} over $\K$: $$P_{Betti}(X;\K):=\sup\{P_{Betti}(\Delta): \Delta \text{ is a triangulation of } X\}.$$
\end{defin}

Analogously we can define the {\em homology splitting probability} $P_{Hom}(\Delta;\K)$ of $\Delta$ over $\K$. Clearly $P_{Betti}(\Delta) \leq P_{Hom}(\Delta).$
%


In this paper, we proved that $P_{Hom}(\Delta)=P_{Betti}(\Delta)$, if $\Delta$ is the triangulation of a manifold
, restricting to standard decompositions for which the intersection is a manifold, see Theorem \ref{teo:3}.
 
Moreover we proved that $P_{Betti}(X)>0$ if $X$ is orientable, see Theorem \ref{viceversa}. Moreover we showed, for instance, that for the given triangulation of the Klein bottle, $P_{Betti}(K;\K)=P_{Hom}(K;\K)=0$, if $\chara(\K) \neq 2$.

Focusing the attention only on the standard decompositions induced by the removal of a single facet, we define the {\em facet splitting probability} of $\Delta$: $$P_{Facet}(\Delta):=\frac{|\mathcal{B}_{\mathcal{F}(\Delta)}|}{|\mathcal{F}(\Delta)|},$$ where $\mathcal{B}_{\mathcal{F}(\Delta)}$ is the collection of standard decompositions of $\Delta$ induced by the removal of a single facet that are Betti splitting.

\begin{oss}
Let $\Delta$ be a $d$-manifold. By Theorem \ref{viceversa}, we proved that $\Delta$ is orientable if and only if $P_{Facet}(\Delta;\K) \neq 0$ for every field $\K$. In this case, $P_{Facet}(\Delta;\K)=1$. Moreover we have that if $\Delta$ is not orientable, then $P_{Facet}(\Delta;\K) \neq 0$ if and only if $\chara(\K)=2$. Also in this case $P_{Facet}(\Delta;\K)=1$.
\end{oss}

We are now able to state precise problems:

\noindent {\bf Problem 1}: Let $X$ be a non-orientable $d$-manifold without boundary. Is it true that $P_{Betti}(X,\K)=0$, if $\K$ is a field with $\chara(\K) \neq 2$? Is it true that every triangulation of $X$ is {\em not} trivially decomposable?

\noindent {\bf Problem 2}: Let $X$ be the mod 3 Moore space. Is it true that $P_{Betti}(X,\K)=0$, if $\K$ is a field with $\chara(\K) \neq 3$?

\noindent {\bf Problem 3}: Let $X$ be the dunce hat. Is it true that $P_{Betti}(X,\K)=0$, for every field $\K$?

\section*{Acknowledgments}
This work has been partially supported by the US National Science Foundation under grant number IIS-1116747.
The authors wish to thank Leila De Floriani, Emanuela De Negri and Maria Evelina Rossi for their helpful comments and suggestions.

\end{document}

%% file: table1-nuova.tex
\begin{table}[!htb]
    \centering
    \caption{Relevant properties of the considered simplicial complexes.} 
        \label{tab:examples}
        \vspace{.5em}
    \begin{tabular}{|c||c|c|c|}
      \hline
        \multirow{2}{*}{$\Delta$} & \multirow{2}{*}{Manifold} & \multirow{2}{*}{Orientable}  & \multirow{2}{*}{$\K$ s.t. $\widetilde{\beta}_d(\Delta;\K)\neq0$}  \\
                &       &  &        \\\hhline{|=#=|=|=|}

     %

     $\RP^{2n+1}$
              & \ding{51}
              & \ding{51}
              & any $\K$
              \\ \hline
              
     $L(3,1)$
               & \ding{51}
               & \ding{51}
               & any $\K$
               \\ \hline

     $\RP^{2n}$
               & \ding{51}
               & \ding{55}
               & $\K$ with $\chara(\K)=2$
               \\ \hline

     $K$
              & \ding{51}
              & \ding{55}
              & $\K$ with $\chara(\K)=2$
              \\ \hline

     $M$
              & \ding{55}
              & $-$
              & $\K$ with $\chara(\K)=3$
              \\ \hline

     $D$
              & \ding{55}
              & $-$
              & none
              \\ \hline
    \end{tabular}


\end{table}

%% file: table2-nuova.tex
\begin{table}[!htb]
    \centering
    \caption{Statistics of the computation.} 
        \label{tab:2}
        \vspace{.5em}
    \begin{tabular}{|c||c|c|c|c|}
      \hline
        \multirow{2}{*}{$\Delta$} & \multirow{2}{*}{$n_V$} & \multirow{2}{*}{$|\mathcal{F}(\Delta)|$}  & Number of & Time \\
                &       &  &  decompositions &  in seconds    \\\hhline{|=#=|=|=|=|}

     $\RP^{2}$
              & 6
              & 10
              & 511
              & 0.07
              \\ \hline

     $K$
              & 8
              & 16
              & 32767
              & 7.47
              \\ \hline

     $D$
              & 8
              & 17
              & 65535
              & 14.11
              \\ \hline

     $M$
              & 9
              & 19
              & 262143
              & 68.44
              \\ \hline
    \end{tabular}


\end{table}